\documentclass[12pt, reqno]{amsart}
\usepackage{amsmath, amsthm, amscd, amsfonts, amssymb, graphicx, color, mathrsfs}
\usepackage[bookmarksnumbered, colorlinks, plainpages]{hyperref}
\usepackage[all]{xy} 
\usepackage{slashed}
\usepackage{cite}
\usepackage{xcolor}
\usepackage{listings}
\usepackage{xcolor}
\usepackage{comment}

\newtheorem{theorem}{Theorem}[section]

\theoremstyle{definition}
\newtheorem{definition}[theorem]{Definition}

\theoremstyle{remark}
\newtheorem{remark}[theorem]{Remark}
\numberwithin{equation}{section}

\begin{document}
\setcounter{page}{1}

\color{darkgray}{
\noindent \centering
{\small   }\hfill    {\small }\\
{\small }\hfill  {\small }}

\centerline{}

\centerline{}


\title[Hyperbolic partial differential equations ]{Hyperbolic partial differential equations with complex characteristics on Fourier Lebesgue spaces}

\author[D. Cardona]{Duv\'an Cardona$^{1,*}$}
\address{
  Duv\'an Cardona:
  \endgraf
  Department of Mathematics: Analysis, Logic and Discrete Mathematics
  \endgraf
  Ghent University, Belgium
  \endgraf
  {\it E-mail address} {\rm duvanc306@gmail.com, duvan.cardonasanchez@ugent.be}
  }

\author[W.O. Denteh]{William Obeng-Denteh$^{2,*}$}
\address{
  William Obeng-Denteh
  \endgraf
  Department of Mathematics
  \endgraf
  Kwame Nkrumah University of Science and Technology, (KNUST)- Ghana.
  \endgraf
  {\it E-mail address} {\rm wobengdenteh@gmail.com}
  }

\author[F. Opoku]{Frederick Opoku$^{3,*}$}
\address{
  Frederick Opoku
  \endgraf
  Department of Mathematics
  \endgraf
  Kwame Nkrumah University of Science and Technology, (KNUST)- Ghana.
  \endgraf
  {\it E-mail address} {\rm frederick@aims.edu.gh}
  }

\thanks{Duv\'an Cardona has been  supported  by the FWO  Odysseus  1  grant  G.0H94.18N:  Analysis  and  Partial Differential Equations and by the Methusalem programme of the Ghent University Special Research Fund (BOF)
(Grant number 01M01021), by the FWO Fellowship
Grant No 1204824N and by the FWO Grant K183725N of the Belgian Research Foundation FWO. Frederick Opoku has been supported by Kwame Nkrumah University of Science and Technology, (KNUST) of Ghana.
}

\begin{abstract} The aim of this paper is to establish well-posedness properties  for hyperbolic PDEs on Fourier Lebesgue spaces. We consider hyperbolic operators with complex characteristics. Since our approach comes from harmonic analysis, we establish boundedness properties of Fourier integral operators with complex-valued phase functions on Fourier Lebesgue spaces, Besov spaces and Triebel-Lizorkin spaces. Indeed, these classes of operators serve as propagators of the considered PDE problems. In terms of the boundedness properties, we prove new results in the case where the canonical relation of the operator is assumed to satisfy the {\it spatial smooth factorization condition}. 
\end{abstract} \maketitle
\allowdisplaybreaks
\tableofcontents

\section{Introduction}
\subsubsection{Main results} Advancing in methods for the analysis of hyperbolic partial differential equations with complex characteristics, the aim of this work is to investigate the well-posedness of these models on Fourier Lebesgue spaces. Our results are extensions of the ones for $L^p$-Sobolev spaces obtained by  Ruzhansky \cite[Chapter 5]{Ruz2001}. See also Melin and Sj\"ostrand \cite{melin1976fourier} and Tr\`{e}ves \cite {treves1980introduction}.

Our approach coming from harmonic analysis will use the boundedness properties of Fourier integral operators with complex phases, which are the propagators or, solution operators for these equations. The general prototype will be hyperbolic PDEs associated with operators of the form (see e.g. Tr\`{e}ves \cite {treves1980introduction})
\begin{align}\label{eq:5.4.10:intro}
P = D_t^{\,m} + \sum_{j=1}^{m} P_j(t,x,D_x) D_t^{\,m-j},
\end{align}
on a smooth compact manifold $X$  of dimension $d.$ 

Of strong interest by themselves are Fourier integral operators with complex-valued phase functions \cite{melin1976fourier} which arise naturally in a wide range of analytical problems, particularly in the study of hyperbolic partial differential equations and microlocal analysis. In contrast to the classical theory of Fourier integral operators with real-valued phase functions, complex phases emerge in settings where the real-phase framework becomes inoperable, in particular, when one is interested in obtaining sharp boundedness results. A striking recent development appears in the analysis of the wave equation associated with Hörmander sub-Laplacians on Heisenberg-type and Métivier groups, as observed by Martini and Müller \cite{MMNG2023}, where the presence of complex-valued phases is intrinsic to the problem.

The main goal of this paper is the analysis of the following Cauchy problem for the hyperbolic operator $P$ as defined in \eqref{eq:5.4.10:intro}:
\begin{equation}\label{eq:5.4.12:intro}
\begin{cases}
Pv = 0, & t\in[0,T], \\
\partial_t^{\,l} v(0,x) = f_l(x), & 0\le l \le m-1 .
\end{cases}
\end{equation} 
Using the machinery of FIOs with complex phases, the boundedness results established in \cite{CardonaObengDentehOpoku} and the ones developed in this work in Theorem \ref{Main:Theorem},  we prove the following result.  We denote by $\mathcal{F}L^p_s$ the Fourier Lebesgue space with the norm $\|f\|_{\mathcal{F}L^p_s}=\|\langle\xi\rangle^s\hat{f}\|_{L^p}$, see Section \ref{Preliminaries}.
\begin{theorem}\label{XXX:intro}
Let $P$ be a classical pseudo-differential operator of order $m$ of the form\eqref{eq:5.4.10:intro}
on a smooth compact manifold $X$  of dimension $d.$ Assume that $P$ has simple characteristics $\tau_j,$ satisfying
$\operatorname{Im}\tau_j (t,x,\xi) \geq 0$ in $[0,T]\times (T^*X\setminus\{0\})$. 
\begin{itemize}
    \item Let $1\le p \le \infty$, and set
$ 
\alpha_p \geq d\left|\frac{1}{p}-\frac{1}{2}\right|.
$ 
Let $\alpha \in \mathbb{R}$. Let the functions $f_l \in \mathcal{F}L^{p}_{\alpha+\alpha_p-l}$ be compactly supported for all
$0\le l \le m-1$. Then, for a fixed $t$, the solution $v$ of the Cauchy problem
\eqref{eq:5.4.12:intro} satisfies
$ 
v(t,\cdot)\in \mathcal{F}L^{p}_{\alpha}.
$
\item Moreover, let $1\le q < p\le \infty,$ and define
$ 
\alpha_{pq}>d\bigg|\dfrac{1}{q}-\dfrac{1}{2}\bigg|+d\bigg(\dfrac{1}{q}-\dfrac{1}{p}\bigg).
$ 
Let the functions $f_l \in \mathcal{F}L^{p}_{\alpha+\alpha_{pq}-l}$ be compactly supported for all
$0\le l \le m-1$. 
Then, the fixed-time solution
$v(t,\cdot)$ of the Cauchy problem~\eqref{eq:5.4.12:intro} satisfies
$ 
v(t,\cdot)\in \mathcal{F}L^{q}_{\alpha}.
$ 
\end{itemize}
\end{theorem}
\subsubsection{Sharpness and regularity of the solution in Besov and in Triebel-Lizorkin spaces} For $L^p$-Sobolev spaces, to our knowledge, Theorem \ref{XXX:intro} has been established by Ruzhansky \cite[Chapter 5]{Ruz2001} with a sharp loss of regularity $$\alpha_p = (d-1)\left|\frac{1}{p}-\frac{1}{2}\right|,$$ with a similar sharp-endpoint loss of regularity in the $L^p$-$L^q$-setting. We also provide information about the membership of the solution $v(t,\cdot)$ to Besov and on Triebel-Lizorking spaces to get information about the regularity of the solution, see Remark \ref{Fredferick:Remark}. Also, in view of the results in Nicola \cite{Nicola2010} the loss of regularity $$\alpha_p = d\left|\frac{1}{p}-\frac{1}{2}\right|,$$ in Theorem \ref{XXX:intro} is sharp.
\subsubsection{Methodology and bibliographic discussion}
As mentioned before, propagators to the class of the hyperbolic PDEs in Theorem \ref{XXX:intro} are Fourier integral operators with complex phases.  The extension of the theory of Fourier integral operators due to H\"ormander and Duistermaat to the case of complex-valued phase function was systematically developed by Melin and Sj\"ostrand in \cite{melin2006fourier} motivated by the construction of parametrices, or fundamental solutions for operators of principal type with non-real principal symbols, see for instance \cite{melin1976fourier}. The Fourier integral operators in the sense of H\"ormander \cite{hormander1963linear,treves1980introduction}, in a simplified local version, are operators of the form 
\begin{align}
    Tf(x)=\int e^{2\pi i \Phi(x,\eta)}\sigma(x,\eta)\mathscr{F}f(\eta)\,d\eta.\label{q}
\end{align}
The phase function $\Phi(x,\eta)$ in (\ref{q}) is assumed real-valued, smooth on $\eta\neq 0$, and positively homogeneous of degree $1$ in $\eta\neq0$. Moreover, the symbol $\sigma:=\sigma(x,\eta)$ belongs to the H\"ormander class $S^m_{1,0}$ of order $m\in \mathbb{R}$, that is, it satisfies that, \[ \forall (x,\eta),\quad|\partial^{\alpha}_{\eta}\partial^{\beta}_x\sigma(x,\eta)| \lesssim \langle \eta \rangle^{m-|\alpha|}, \] where $\langle \eta \rangle=(1+|\eta|^2)^{1/2}.$ Boundedness in $L^2(\mathbb{R}^d)$ and on $L^p(\mathbb{R}^d)$ spaces of these operators have been widely studied, see e.g in \cite{sogge2017fourier} and in references therein. As a basic results, we know that under non-degeneracy condition 
\begin{align}\label{q11}
   \bigg| \operatorname{det}\bigg(\dfrac{\partial^2\Phi}{\partial_x{_i}\partial_\eta{_i}}\bigg|_{(x,\eta)}\bigg)\bigg|>\delta>0, \forall (x,\eta)\in \mathbb{R}^{2d},
\end{align}
the operator $T$ is $L^2-$bounded for $m=0,$ see \cite{hormander1971fourier}. Also, they are $L^p-$bounded, if $1<p<\infty,$ and if the order $m$ of $\sigma(x,\eta)$ is negative and it satisfies 
\begin{align}\label{q1v}m\le -(d-1)\bigg|\dfrac{1}{2}-\dfrac{1}{p}\bigg|,\end{align} see \cite{seeger1991regularity}. In \cite{Nicola2010}, the author has studied the action of an operator $T$ as above, on the spaces $\mathcal{F}L^p$ of tempered distributions whose Fourier transform are in $L^p$  with the norm $\|f\|_{\mathcal{F}L^p}=\|\hat{f}\|_{L^p}$. There, it was shown that $T$ is bounded as an operator $\mathcal{F}L^p\longrightarrow \mathcal{F}L^p,$ $1\le p\le \infty$, if $m\le -d \bigg|\dfrac{1}{2}-\dfrac{1}{p}\bigg|$ in the non-degenerate case.  This is similar to (\ref{q1v}), but with the difference of one unit in the dimension. Surprisingly, this threshold was shown to be sharp in any dimension $d\ge 1,$ even for the phases linear with respect to $\eta$; see \cite{cordero2009boundedness}. We also remark that we have generalised the result due to \cite{Nicola2010}, to the setting of complex phases. In this paper we present the boundedness of Fourier integral operators $T$ on Besov and Triebel--Lizorkin spaces under the assumption that the associated canonical relation satisfies the spatial smooth factorisation condition in \cite{CardonaObengDentehOpoku}, and also use it to establish the boundedness on Fourier--Lebesgue spaces for different choices of the Lebesgue indicies $p$ and $q.$
 This is a continuation of our study \cite{CardonaObengDentehOpoku}, where we dealt with boundedness of Fourier integral operators with complex-valued phase on Fourier Lebesgue spaces. Recently, in \cite{Lu} the author has established sharp embedding properties of Fourier--Lebesgue spaces. Motivated by this result, we investigate the boundedness of Fourier integral operators $T$ on Besov and Triebel--Lizorkin spaces by suitably modifying the arguments developed in \cite{CardonaObengDentehOpoku}. Our approach relies on the embedding results into Fourier--Lebesgue spaces obtained in \cite{Lu}, which hold for a wide range of indices and, in particular, allow $0<r\le\infty$. The methods and techniques developed in this paper combine ideas originating from the work of Ruzhansky \cite{Ruz2001}, Laptev, Safarov, and Vassiliev \cite{Laptev:Safarov:Vassiliev}, Rodino, Nicola, and Cordero \cite{CNR2009}, Nicola \cite{Nicola2010}, as well as from joint work of the first author with Ruzhansky. These tools are adapted and refined to address the boundedness of Fourier integral operators with complex-valued phase functions under the spatial smooth factorization condition.
\subsubsection{Structure of the paper} 
This paper is organised as follows. 
\begin{itemize}
    \item In Section \ref{Preliminaries}, we collect the basic definitions and notations that will be used throughout the paper, and we introduce the Kohn--Nirenberg symbol classes of type $(1,0)$, which play a central role in the estimation of symbols. 
    \item In Section \ref{Hype:PDE}, we investigate the boundedness properties of Fourier integral operators with complex phases and we investigate the well posedness of hyperbolic partial differential equations in those spaces of functions of interest in this manuscript.
\end{itemize}

\section{Preliminaries}\label{Preliminaries}
\subsubsection{Notations} For non–negative quantities $A$ and $B$, the notation $A\lesssim B,$ means that
there exists a constant $C>0$, independent of the relevant parameters, such that
$A\le CB$. We write $\mathscr{F}$ to represent the Fourier transform, defined for suitable functions and put $\mathscr{F}^{-1}$ for its inverse Fourier transform. We denote by $\mathscr{F}L^p$ the Fourier Lebesgue space with the norm $\|f\|_{\mathcal{F}L^p}=\|\hat{f}\|_{L^p}$. We denote $r'$ as the quantity such that $ \dfrac{1}{r}+\dfrac{1}{r'}=1$; denote $r'=\infty$ for $r'=\dfrac{r}{r-1}$ .  We denote by $X\hookrightarrow Y$ if $X\subseteq Y$ and the inclusion map is continuous. 
\subsection{Besov spaces $B^{s}_{p,q}$ and Triebel-Lizorkin Spaces $F^{s}_{p,q}$} Let $0<p,q\le \infty, s\in \mathbb{R}.$ Let $\psi:\mathbb{R}^d\longrightarrow [0,1]$ be a cut-off function supported on the ball $B(0,2)=\{|\xi|<2\}$. It satisfies that $\psi(\xi)=1$ for $|\xi|\le 1$ and $\psi(\xi)=0$ for $|\xi|\ge 2$. We denote $\varphi(\xi)=\psi(\xi)-\psi(2\xi),$ and $\varphi_j(\xi)=\varphi(2^{-j}\xi)$ for $1\le j, j\in \mathbb{Z}, \varphi_0(\xi)=1-\sum_{j\ge 1}\varphi_j(\xi).$ Denote $\Delta_j:=\mathscr{F}^{-1}\varphi_j\mathscr{F}$. We say that $\{\Delta_j\}_{j\ge 0}$ are the dyadic decomposition operators. The Besov spaces $B^s_{p,q}$ and the Triebel spaces $F^s_{p,q}(p<\infty)$ are defined as follows:\\
\begin{align}
    B^s_{p,q}(\mathbb{R}^d)&=\bigg\{f\in \mathcal{S}^{'}(\mathbb{R}^d):\|f\|_{B^s_{p,q}(\mathbb{R}^d)}=\bigg\|2^{js}\|\Delta_jf\|_{L^p_x}\bigg\|_{\ell^q_{j\ge 0}(\mathbb{N}_0)} <\infty\bigg\},\\
    F^s_{p,q}(\mathbb{R}^d)&=\bigg\{f\in \mathcal{S}^{'}(\mathbb{R}^d):\|f\|_{F^s_{p,q}(\mathbb{R}^d)}=\bigg\|2^{js}\|\Delta_jf\|_{\ell^q_{j\ge 0}}\bigg\|_{L^p_x(\mathbb{R}^d)} <\infty\bigg\}.
\end{align}
The spaces $B^s_{p,q}(\mathbb{R}^d)$ and $F^s_{p,q}(\mathbb{R}^d)$
are independent of the particular choice of the smooth dyadic decomposition of unity appearing in their definition. They are quasi-Banach spaces (Banach spaces for $p,q  \ge 1)$, and $\mathcal{S}(\mathbb{R}^d)\hookrightarrow B^s_{p,q}(\mathbb{R}^d)\hookrightarrow \mathcal{S'}(\mathbb{R}^d)$
where the first embedding is dense if $0<p,q<\infty$.
There is a parallel approach when interchanging in the above norm the $L^p$
and $\ell^q$ norm, this leads to the scale of Triebel–Lizorkin spaces $F^s_{p,q}(\mathbb{R}^d)$, see \cite{triebel2012structure}.

\subsection{Fourier Integral Operators (FIOs)} 
Let $\mathcal{S}(\mathbb{R}^d)$ be the
Schwartz space of all complex-valued rapidly decreasing infinitely differentiable functions on $\mathbb{R}^d$ . By $\mathcal{S'}(\mathbb{R}^d)$ we denote 
the space of all tempered distributions on $\mathbb{R}^d.$  
 The Fourier transform and its inversion formula of $f\in\mathcal{S}(\mathbb{R}^d)$ are defined by
\begin{align}
 \mathscr{F}f(\eta)&=\int_{\mathbb{R}^d} e^{-2\pi i x\cdot\eta} f(x)\,dx;\\
 \mathscr{F}^{-1}f(x)&=\int_{\mathbb{R}^d} e^{2\pi i x\cdot\eta} \mathscr{F}f(\eta)\,d\eta,
\end{align}
and extended to $\mathcal{S}'(\mathbb{R}^d)$ by duality.
For $1\le p\le\infty,$ and $s\in\mathbb{R}$, the Fourier--Lebesgue space
$\mathcal{F}L^p_s(\mathbb{R}^d)$ consists of all tempered distributions
$f\in\mathcal{S}'(\mathbb{R}^d)$ such that
\[
\|f\|_{\mathcal{F}L^p_s}
=
\|\langle\eta\rangle^s \mathscr{F}f(\eta)\|_{L^p(\mathbb{R}^d)}<\infty,
\quad
\langle\eta\rangle=(1+|\eta|^2)^{1/2}.
\]
Let $\sigma:=\sigma(x,\eta)\in S^m_{1,0}$ be a
symbol compactly supported in the $x$ variable and let  $\Phi$ to be a phase function of positive type. These classes of functions are defined by the inequalities (\ref{INQ}) below. The Fourier integral operator $T$ associated with $\Phi$ and $\sigma$ is
 defined by
\begin{equation}
  Tf(x)
=
\int_{\mathbb{R}^d} e^{2\pi i \Phi(x,y,\eta)}\sigma(x,\eta)\mathscr{F}f(\eta)\,d\eta.  
\end{equation}
Under the standard non–degeneracy assumptions on $\Phi$, the operator $T$ defines an operator mapping $C_c^\infty(Y)$ into $\mathcal{D}'(X)$ and extends by density
to various function spaces.
We analyse the boundedness of $T$ for the symbol $\sigma:=\sigma(x,\eta)\in C^{\infty}(\mathbb{R}^{2d})$ in the symbol class $S^m_{1,0}$ with $m\in \mathbb{R}$ which satisfy the following estimate
\begin{align}\label{INQ}|\partial^{\alpha}_{\eta}\partial^{\beta}_x\sigma(x,\eta)| \lesssim \langle \eta \rangle^{m-|\alpha|},\quad \forall (x,\eta)\in \mathbb{R}^{2d}, \quad\text{ where } \langle \eta \rangle=(1+|\eta|^2)^{1/2}.\end{align}
In order to introduce Fourier integral operators with complex phases, we require  $\Phi$ to be a phase function of positive type; this means that:
    \begin{itemize}
        \item  $\operatorname{Im} \Phi(x,y,\eta)\ge 0$,
        \item $\Phi(x,y,\eta)$ has no critical points on its domain: $\partial_{\theta} \Phi(x,y,\eta)\neq 0.$
        \item $\Phi(x,y,\eta)$ is homogeneous of degree one in $\eta$ i.e, $\forall \lambda>0, \Phi(x,y,\lambda\eta)=\lambda\Phi(x,y,\eta).$
        \item $\{\partial_\eta\Phi(x,y,\eta)=0\}$ is smooth. This means that if $\partial_\eta\Phi(x,y,\eta)=0,$ then the  vectors $d\frac{\partial\Phi}{\partial\eta_j}$ are linearly independent over $\mathbb{C}.$
    \end{itemize}
    The associated canonical relation $\mathcal{C}\subset \widetilde{T^*X\times T^*Y}$ is defined by
\[
\mathcal{C}
=
\{(x,\nabla_x\Phi(x,y,\eta),\,y,-\nabla_y\Phi(x,y,\eta)) :
\nabla_\eta\Phi(x,y,\eta)=0\},
\] where $\widetilde{M}$ denotes the almost analytic continuation of a real smooth manifold $M,$ see \cite{Melin-Sjostrand1976}. It is important to mention that the analytical properties of a Fourier integral operator are encode by the geometric properties of its canonical relation. To establish Fourier--Lebesgue boundedness for Fourier integral operators with complex phases, we adopt the \emph{spatial smooth factorization condition} (SSFC); for details, we refer the reader to \cite{Nicola2010}.
\begin{definition}[Spatial smooth factorization condition (SSFC)]\label{SFC}
We assume that there exists an integer $0\le \varkappa\le d$ such that, for some 
$\tau \in \mathbb{R},$ and for every $(x_0,\eta_0)\in\Lambda_\tau$ with
$\eta_0\in\mathbb{S}^{d-1}$, there exists an open neighbourhood $\Omega$ of $x_0$
and an open neighbourhood $\Gamma'\subset\mathbb{S}^{d-1}$ of $\eta_0$, with
$\Omega\times\Gamma'\subset\Lambda_\tau$, satisfying the following properties.
\begin{itemize}
    \item {{(SSFC)}} For every $\eta\in\Gamma'$ there exists a smooth fibration of $\Omega$,
depending smoothly on $\eta$, whose fibres are affine subspaces of codimension
$\varkappa$, such that
\[
\nabla_x\Phi_\tau(\cdot,\eta)
\quad\text{is constant along each fibre}.
\]
\item In addition, we also  assume that $\Phi_\tau$ holds the non-degeneracy condition,
$$\operatorname{det}\partial_y\partial_{\eta}(\operatorname{Re}\Phi(y,\eta)+\tau\operatorname{Im}\Phi(y,\eta))\neq 0.$$ 
\end{itemize}
To simplify the terminology, if a phase function $\Phi$ can be associated with a phase function $\Phi_\tau$ that satisfies both properties in Definition \ref{SFC} for some $\tau \in \mathbb{R}$, we say that $\Phi$ satisfies the smooth spatial factorization condition of rank $\varkappa$.
\end{definition}
\subsection{Fourier Lebesgue spaces vs $B^{s}_{p,q}$ and $F^{s}_{p,q}$ spaces}
In the proof of Theorems~\ref{Main:Theorem}, we make essential use of boundedness results for Fourier integral operators with complex-valued phase functions acting on Fourier--Lebesgue spaces. In Theorems~\ref{Thm2.4} and~\ref{Thm2.5}, the relevant embeddings into Fourier--Lebesgue spaces are established for a wide range of indices, allowing $0<r\le\infty$. 
However, in the proofs of Theorems~\ref{Main:Theorem}, we restrict attention to the range $1\le r\le\infty,$ in order to obtain boundedness of the Fourier integral operator $T$ from Besov spaces to Fourier--Lebesgue spaces and from Triebel--Lizorkin spaces to Fourier--Lebesgue spaces. This restriction is natural, since the operator
\[
T:\mathscr{F}L^{r}\longrightarrow \mathscr{F}L^{r}
\]
is bounded whenever
\[
m\le -\varkappa\bigg|\frac{1}{r}-\frac{1}{2}\bigg|,
\qquad 1\le r\le\infty.
\]
The arguments employed in \cite{CardonaObengDentehOpoku} are based on this condition and make use of it to establish the boundedness of the adjoint operator $T^{*}$ via duality when $1<r<\infty$. 
Although the embedding results in Theorems~\ref{Thm2.4} and~\ref{Thm2.5} remain valid in the range of quasi-Banach spaces, the boundedness results of Theorems~\ref{Main:Theorem}  are therefore formulated and proved only for $1\le r\le\infty.$ Next, we introduce some embedding properties about Fourier--Lebesgue spaces on Besov spaces and Triebel-Lizorkin spaces.
\begin{theorem}[$B^s_{p,q} \label{Thm2.4} \hookrightarrow \mathscr{F}L^r$] Let $0<p,q,r\le \infty, s\in\mathbb{R} \label{ab}
$. Then $B^s_{p,q}  \hookrightarrow \mathscr{F}L^r$ holds if and only if $1/p+1/r\ge 1, p\le 2$ and one of the following conditions is satisfied: 
\begin{enumerate}
    \item $q\le r, s\ge d(1/r+1/p-1);$
    \item $q> r, s> d(1/r+1/p-1)$.
\end{enumerate}
\end{theorem}
As for Triebel-Lizorkin spaces we have these properties.
\begin{theorem}[$F^{s}_{p,q}\label{Thm2.5} \hookrightarrow \mathscr{F}L^{r}$]
Let $0 < p < \infty$, $0 < q, r \le \infty$, and $s \in \mathbb{R}$.  Then $
F^{s}_{p,q} \hookrightarrow \mathscr{F}L^{r}$ \label{1}
holds if and only if $p \le 2,1/p + 1/r \ge 1,$
and one of the following conditions is satisfied:
\begin{enumerate}
  \item $r < p, 
  s > d\Big(\frac{1}{p} + \frac{1}{r} - 1\Big);$
\item $r \ge p$, $r = p' < q$, 
  $s > d\Big(\frac{1}{p} + \frac{1}{r} - 1\Big);$
  \item $r \ge \max\{p,q\},$ $s \ge d\Big(\frac{1}{p} + \frac{1}{r} - 1\Big);$
  \item $p \le r < p'$ , $
  s \ge d\Big(\frac{1}{p} + \frac{1}{r} - 1\Big).$
\end{enumerate}
\end{theorem}
The embedding properties above will be used in the next section.

\section{Proof of Theorem \ref{XXX:intro}}\label{Hype:PDE}
\subsection{Boundedness properties of FIOs}

The following is the main theorem of this subsection. This result  and also the main theorem in \cite{CardonaObengDentehOpoku} will be used in the proof of Theorem \ref{XXX:intro}. We note that in some sense the boundedness properties between Fourier Lebesgue spaces and Besov or Triebel-Lizorkin spaces are a direct consequence of the embedding properties presented in the previous section. However, we present it since Besov spaces and Triebel-Lizorkin spaces provide information about the regularity $s,$ or $-s,$ of functions, well in the domain or in the co-domain of the operator. 

\begin{theorem}\label{Main:Theorem} Let $T\in I^m(\mathbb{R}^d,\mathbb{R}^d; \mathcal{C})$ be a Fourier integral operator associated with a complex canonical relation $\mathcal{C}$, locally parametrized by a phase function $\Phi$ of positive type, and satisfying the spatial smooth factorization condition (SSFC) of rank $\varkappa$,\,$0\le \varkappa\le d$ (see Definition~\ref{SFC}). The following boundedness properties hold. 
\begin{itemize}
    \item[1.]   If $m\in \mathbb{R}$ is such that, \[ m< -\varkappa \bigg|\dfrac{1}{q}-\dfrac{1}{2}\bigg|-d\bigg(\dfrac{1}{q}-\dfrac{1}{p}\bigg),\] with $1\le q< p<\infty,$ then the operator $T:\mathscr{F}L^p\longrightarrow \mathscr{F}L^q ,$ extends to a bounded operator, that is,  \[ \exists C>0,\, \forall f\in \mathscr{F}L^p,\,\, \|Tf\|_{\mathscr{F}L^q} \lesssim \|f\|_{\mathscr{F}L^p}.\]  
    \item[2.] For every $0<p,q<\infty, 1\le r\le \infty,\text{ and } s\in \mathbb{R}$,  the operator $T:B^s_{p,q} \longrightarrow \mathscr{F}L^r$ extends to a bounded operator, that is \[\exists C>0,\, \forall f\in B^s_{p,q}, \, \|Tf\|_{\mathscr{F}L^r}\lesssim \|f\|_{B^s_{p,q}},\] 
    provided that $m\in \mathbb{R}$ satisfies that
    $$ m\le - \varkappa \bigg|\dfrac{1}{r}-\dfrac{1}{2}\bigg|,\, 1/p+1/r\ge 1,\quad p\le 2,$$ and one of the following conditions is satisfied:
    \begin{enumerate}
    \item $q\le r, s\ge d(1/r+1/p-1);$
    \item $q> r, s> d(1/r+1/p-1)$.
\end{enumerate}Moreover, if $1<r<\infty,$ then $T: \mathscr{F}L^{r'} \to B^{-s}_{p',q'}$ extends to a bounded operator.
\item[3.] For every $0 < p < \infty$, $0 < q \le \infty$, $1\le r \le \infty$ and $s \in \mathbb{R}$,  the operator $T:F^s_{p,q} \longrightarrow \mathscr{F}L^r$ extends to a bounded operator, that is, \[\exists C>0,\, \forall f\in F^s_{p,q}, \, \|Tf\|_{\mathscr{F}L^r}\lesssim \|f\|_{F^s_{p,q}},\] if and only if $p \le 2,1/p + 1/r \ge 1,$
and one of the following conditions is satisfied:
\begin{enumerate}
  \item $r < p, 
  s > d\Big(\frac{1}{p} + \frac{1}{r} - 1\Big);$
\item $r \ge p$, $r = p' < q$, 
  $s > d\Big(\frac{1}{p} + \frac{1}{r} - 1\Big);$
  \item $r \ge \max\{p,q\},$  $s \ge d\Big(\frac{1}{p} + \frac{1}{r} - 1\Big);$
  \item $p \le r < p'$ , $
  s \ge d\Big(\frac{1}{p} + \frac{1}{r} - 1\Big)$,
\end{enumerate}
and  in all the cases above the order condition $m\le -\varkappa \bigg|\dfrac{1}{r}-\dfrac{1}{2}\bigg|$, holds.\\
Moreover, if $1<r<\infty,$ then $T: \mathscr{F}L^{r'} \to F^{-s}_{p',q'}$ extends to a bounded operator.
\end{itemize}    
\end{theorem}

\begin{proof}[Proof of Theorem \ref{Main:Theorem}]
Let us consider $T\in I^m(\mathbb{R}^d,\mathbb{R}^d;\mathcal{C})$ with the canonical relation $\mathcal{C},$ (locally) parametrized by a complex-valued phase function $\Phi$ satisfying the spatial smooth factorization condition (SSFC) with rank $\varkappa$ where $0\le \varkappa \le d$. 
For the proof of (1) let us fix $s\in \mathbb{R},$ and we let $(1-\Delta)^{s/2}(1-\Delta)^{-s/2}=id.$
We decompose $Tf$ using the Bessel potential operator $(1-\Delta)^{-s/2}$ as 
\begin{align}
 Tf=T(1-\Delta)^{s/2}(1-\Delta)^{-s/2}f.\label{FLS}
 \end{align}
 By taking the Fourier-Lebesgue norm on (\ref{FLS}) one has that
\begin{align*}
     \|Tf\|_{\mathscr{F}L^q}  &= \|T(1-\Delta)^{s/2}(1-\Delta)^{-s/2}f\|_{\mathscr{F}L^q}\\
  &\le \|T(1-\Delta)^{s/2}\|_{\mathscr{F}L^q}\|(1-\Delta)^{-s/2}f\|_{\mathscr{F}L^q}.
\end{align*}
The order of $T$ and $(1-\Delta)^{s/2}$ are $m$ and $s$ resectively.  Since, the composition of a Fourier integral operator with a Fourier multiplier, on the right, is again a Fourier integral operator with the same canonical relation \cite{Melin-Sjostrand1976},  the operator $T(1-\Delta)^{s/2}$ is also a Fourier integral operator of order $m+s.$ We have that $T(1-\Delta)^{s/2}\in I^{m+s}(\mathbb{R}^d,\mathbb{R}^d;\mathcal{C})$.
We seek to find the bound of $T(1-\Delta)^{s/2}$ in the Fourier Lebesgue space. That is,
$ T(1-\Delta)^{s/2}: \mathscr{F}L^q\longrightarrow \mathscr{F}L^q,$  is bounded provided $m+s\le-\varkappa\bigg|\dfrac{1}{q}-\dfrac{1}{2}\bigg|.$
Again, we now estimate $(1-\Delta)^{-s/2}f.$ By the definition of the Fourier-Lebesgue norm, we have,
\begin{align*}
    \|(1-\Delta)^{-s/2}f\|_{\mathscr{F}L^q}^q&= \int_{\mathbb{R}^d}|\mathscr{F}[{(1-\Delta)^{-s/2}f](\eta)}|^q\,d\eta
    = \int_{\mathbb{R}^d} \langle \eta \rangle^{-sq}|\mathscr{F}f(\eta)|^q\,d\eta.
    \end{align*}
     By applying H\"older's inequality  with, $\dfrac{1}{r'}+\dfrac{1}{r}=1,$ and $1\le r\le \infty,$ yields 
    \begin{align*}
    \|(1-\Delta)^{-s/2}f\|_{\mathscr{F}L^q}^q   &\le \bigg(\int_{\mathbb{R}^d}\langle \eta \rangle^{-sqr'} \,d\eta\bigg )^{1/r'}\bigg(\int_{\mathbb{R}^d}|\mathscr{F}f(\eta)|^{qr} \,d\eta\bigg )^{1/r}. 
\end{align*}
Let $qr=p\ge 1,$ we have that,
\begin{align*}
   \|(1-\Delta)^{-s/2}f\|_{\mathscr{F}L^q}^q &\le C \bigg(\int_{\mathbb{R}^d}|\mathscr{F}f(\eta)|^{qr} \,d\eta\bigg )^{1/r}\le C \bigg(\int_{\mathbb{R}^d}|\mathscr{F}f(\eta)|^{p} \,d\eta\bigg )^{1/p},
\end{align*}
where 
$C:= \int_{\mathbb{R}^d}\langle \eta \rangle^{-sqr'} \,d\eta < \infty,$ provided that $s>d\bigg(\dfrac{1}{q}-\dfrac{1}{p}\bigg).$ 
Fix $\epsilon >0,$ and set $s=d\bigg(\dfrac{1}{q}-\dfrac{1}{p}\bigg)+\epsilon$. With this choice the operator $(1-\Delta)^{-s/2}$ defines a bounded operator
$(1-\Delta)^{-s/2}:\mathscr{F}L^p \longrightarrow \mathscr{F}L^q.$
This implies that the boundedness of $T(1-\Delta)^{-s/2}$ is guaranteed provided
$ 
m\le -\varkappa\bigg|\dfrac{1}{q}-\dfrac{1}{2}\bigg|-d\bigg(\dfrac{1}{q}-\dfrac{1}{p}\bigg)-\epsilon . 
 $ 
Since $\epsilon>0,\, m<-\varkappa\bigg|\dfrac{1}{q}-\dfrac{1}{2}\bigg|-d\bigg(\dfrac{1}{q}-\dfrac{1}{p}\bigg),$ for $1\le p<q\le \infty.$
Therefore, 
$ 
 \|Tf\|_{\mathscr{F}L^q} \lesssim \|f\|_{\mathscr{F}L^p}\,\, \forall f \in \mathscr{F}L^p.
$ The proof is complete.

Now, let us prove (2). Let $0<p,q< \infty$, $1\le r\le \infty $ and $s\in \mathbb{R}.$ 
     Under the assumptions stated above, we have a continuous embedding 
    $ 
         i: B^s_{p,q}\hookrightarrow \mathscr{F}L^r,
     $ 
     with 
     $ 
         \|if\|_{\mathscr{F}L^r}\le \|f\|_{B^s_{p,q}}.
     $ 
     Since $\Phi$ satisfies the spatial smooth factorization condition of rank $\varkappa,$ and it is associated to the canonical relation $\mathcal{C}$, that is $T\in I^m(\mathbb{R}^d,\mathbb{R}^d, \mathcal{C})$, the Fourier integral operator $T$ is bounded on the Fourier Lebesgue space 
     $ 
         T=T\circ i:\mathscr{F}L^r\longrightarrow \mathscr{F}L^r,
     $ 
     provided that $ m\le -\varkappa\bigg|\dfrac{1}{r}-\dfrac{1}{2}\bigg|. $ 
     That is, we have the inequality $\|Tg\|_{\mathscr{F}L^r}\lesssim \|g\|_{\mathscr{F}L^r}, \, \forall g\in \mathscr{F}L^r.$
     For $f\in B^s_{p,q},$ 
     we have that $Tf=(T\circ i)f,$ such that
     $  \|Tf\|_{\mathscr{F}L^r}=\|T(if)\|_{\mathscr{F}L^r}\lesssim \|if\|_{\mathscr{F}L^r}\lesssim \|f\|_{\mathscr{F}L^r}.$ 
     Hence $  T:B^s_{p,q}\longrightarrow \mathscr{F}L^r,$  is bounded.
     If $1<r<\infty,$ then the dual embedding satisfies 
     $  i^*: \mathscr{F}L^{r'}\longrightarrow B^{-s}_{p',q'},  $ where $r'$ is defined by the identity: $ \dfrac{1}{r}+\dfrac{1}{r'}=1. $ 
     Since $T$ is bounded on $\mathscr{F}L^{r'}$ when $  m\le -\varkappa\bigg|\dfrac{1}{r}-\dfrac{1}{2}\bigg|,$ 
     the composition $i^*\circ T$ yields
     $ T: \mathscr{F}L^{r'} \to B^{-s}_{p',q'}, $  with the estimate $  \|Tf\|_{B^{-s}_{p',q'}}\lesssim \|f\|_{\mathscr{F}L^r{'}}. $
     
For the proof of (3), let $0<p< \infty$, $0<q\le \infty$ ,$1\le r\le \infty $ and $s\in \mathbb{R}.$ 
     Under the assumptions stated above, we have a continuous embedding 
     $ 
         i: F^s_{p,q}\hookrightarrow \mathscr{F}L^r,
     $ 
     with 
      $ 
         \|if\|_{\mathscr{F}L^r}\le \|f\|_{F^s_{p,q}}.
     $ 
     Since the phase $\Phi$ satisfies the spatial smooth factorization condition of rank $\varkappa$, associated with a canonical relation $\mathcal{C}$, with $T\in I^m(\mathbb{R}^d,\mathbb{R}^d, \mathcal{C})$, the Fourier integral operator $T$ is bounded on the Fourier Lebesgue space 
      $ 
         T=T\circ i:\mathscr{F}L^r\longrightarrow \mathscr{F}L^r,
      $ 
     provided that $ m\le -\varkappa\bigg|\dfrac{1}{r}-\dfrac{1}{2}\bigg|, $ 
     that is, $\|Tg\|_{\mathscr{F}L^r}\lesssim \|g\|_{\mathscr{F}L^r}, \, \forall g\in \mathscr{F}L^r.$
     For $f\in F^s_{p,q},$ 
     we have $Tf=(T\circ i)f,$ and 
     $  \|Tf\|_{\mathscr{F}L^r}=\|T(if)\|_{\mathscr{F}L^r}\lesssim \|if\|_{\mathscr{F}L^r}\lesssim \|f\|_{\mathscr{F}L^r}, $ 
     hence  $  T:F^s_{p,q}\longrightarrow \mathscr{F}L^r, $ is bounded. 
     If $1<r<\infty,$ then the dual embedding satisfies 
     $  i^*: \mathscr{F}L^{r'}\longrightarrow F^{-s}_{p',q'},  $  where $r$ and $r'$ satisfy, $ \dfrac{1}{r}+\dfrac{1}{r'}=1. $ 
     Since $T$ is bounded on $\mathscr{F}L^{r'}$ whenever $ m\le -\varkappa\bigg|\dfrac{1}{r}-\dfrac{1}{2}\bigg|, $
     the composition $i^*\circ T$ yields
     $ T: \mathscr{F}L^{r'} \to F^{-s}_{p',q'} , $  with the estimate  $  \|Tf\|_{F^{-s}_{p',q'}}\lesssim \|f\|_{\mathscr{F}L^r{'}}. $ 
     The proof is complete.
\end{proof}

\subsection{Well posedness properties}
In this subsection, we present applications of hyperbolic partial differential equations to the boundedness of Fourier integral operators with complex-valued phase functions. These properties were presented in Theorem \ref{XXX:intro} which we split in the analysis of the two theorems in this section. These applications are analogous to those established in \cite[Theorems~5.4.2 and~5.4.3]{Ruzhansky2001}. There boundedness results were restricted to the range $1<p<\infty$ on Lebesgue spaces $L^p$, the results presented here extend the regularity assumptions to the full range $1\leq p\leq \infty$ on Fourier--Lebesgue spaces $\mathcal{F}L^p$. The theorems stated below are based on applications of Fourier integral operators with complex phases developed in our previous work \cite{CardonaObengDentehOpoku}. The section is divided into two parts. In the first subsection, we treat the case of first-order hyperbolic partial differential operators. In the second subsection, we generalize these results to hyperbolic partial differential operators of any arbitrary order $m$.

\subsubsection{First order hyperbolic partial differential equations}
Let $A = A(x,D_x)$ be a first order pseudo-differential operator in the space of smooth functions on a compact manifold $X$, where
$D_x = - i\partial_x$, and
$D_t = - i\partial_t$.
We first consider the Cauchy problem
\begin{align}\label{AP1}
\begin{cases}
D_t v - A(x,D_x)v = 0, \\
v(0,x) = f(x).
\end{cases}
\end{align}
If the principal symbol $a(x,\xi)$ of $A$ is real-valued, then the Cauchy problem \eqref{AP1} is strictly hyperbolic and well posed. The solution operator $U(t)$ can be regarded as a Fourier integral operator with complex phase and of order zero. By applying Theorem~1.1 of \cite{CardonaObengDentehOpoku} or the main theorem of Nicola \cite{Nicola2010},   for $1 \le p \le\infty$, we have that the following $\mathcal{F}L^p$-regularity properties hold for the
fixed-time solutions of \eqref{AP1}. If $\alpha\in \mathbb{R},$ and 
 $ 
f \in \mathcal{F}L^{p}_{\alpha+d\lvert 1/p-1/2\rvert}(X),
 $ 
then
$ 
v(t,\cdot)\in \mathcal{F}L^{p}_{\alpha},
$ 
for all sufficiently small values of $t$, and the loss of regularity
 $ 
-d\left| \frac{1}{p}-\frac{1}{2}\right|,
 $ 
is sharp, see e.g. \cite{Nicola2010}.
We now consider a more general case. We assume that the principal symbol $a(x,\xi)$ is
complex-valued and satisfies
\begin{align}\label{eq:5.4.2}
\operatorname{Im} a(x,\xi) \ge 0,
\qquad
\forall (x,\xi)\in T^*X\setminus\{0\}.
\end{align}
We also assume that $A$ is a first order classical pseudo-differential operator. Since we are interested in local estimates, we may assume without loss of generality that $X$ is a compact manifold. The question is whether there exists an operator solution $U(t)$ to
\eqref{AP1}, that is, a continuous linear operator on
$\mathcal{D}'(X)$ such that in $X$ one has the initial value problem:
\begin{align}\label{AP2}
\begin{cases}
D_t U - A U = 0, \\
U(0) = I.
\end{cases}
\end{align}
Such an operator $U(t)$ exists for $t\ge 0$, is unique, it is a Fourier integral operator of order zero and with a global complex phase, and it depends smoothly on $t,$ see Laptev, Safarov, and Vassiliev, \cite{Laptev:Safarov:Vassiliev}. Our main theorems in the next section will provide estimates for these problems on Fourier Lebesgue spaces. 

\subsubsection{Higher order hyperbolic partial differential equations}
Let us consider $P$ to be a differential–pseudo-differential operator on $[0,T]\times X,$ of order
$m,$ and of the form:
\begin{align}\label{eq:5.4.10}
P = D_t^{\,m} + \sum_{j=1}^{m} P_j(t,x,D_x) D_t^{\,m-j},
\end{align}
where $X$ is a smooth compact manifold of dimension $d$, and for every
$1\le j\le m$, $P_j(t,x,D_x)$ is a classical pseudo-differential operator of order $j$ on
$X$, depending smoothly on $t$. The principal symbol $p_j(t,x,\xi)$ of $P_j(t,x,D_x)$ is a
smooth complex-valued function on $[0,T]\times (T^*X\setminus\{0\})$, positively
homogeneous of degree $j$ in $\xi$.
The principal symbol $p(t,\tau,x,\xi)$ of $P$ is given by
\begin{equation}\label{eq:5.4.11}
p(t,\tau,x,\xi)
= \tau^{m} + \sum_{j=1}^{m} p_j(t,x,\xi)\,\tau^{m-j}.
\end{equation}
We make the following assumptions on $P$. First, we assume that $P$ has simple characteristics. This means that for any
$(x_0,\xi^0)\in T^*X\setminus\{0\}$ and any $t_0\in[0,T]$, the roots of the
polynomial $p(t_0,\tau,x_0,\xi^0)$ in $\tau$ are distinct. Let $\tau_j$ denote
these roots. By the implicit function theorem, it follows that the functions $\tau_j$ are complex-valued, smooth on $[0,T]\times (T^*X\setminus\{0\})$, and
positively homogeneous of degree one in $\xi$. The principal symbol $p$ of $P$
can therefore be decomposed as
\[
p(t,\tau,x,\xi) = \prod_{j=1}^{m} \bigl(\tau - \tau_j(t,x,\xi)\bigr).
\]
Our second assumption is that
\[
\operatorname{Im}\tau_j \geq  0,
\quad \text{ in } [0,T]\times (T^*X\setminus\{0\}).
\]
We consider the following Cauchy problem for $P$:
\begin{equation}\label{eq:5.4.12}
\begin{cases}
Pv = 0, & t\in[0,T], \\
\partial_t^{\,l} v(0,x) = f_l(x), & 0\le l \le m-1 .
\end{cases}
\end{equation}
The assumption that the characteristics of $P$ are simple implies that $P$ can
be factored as:
\begin{align}
P = L_m \cdots L_1 + R,\label{Lm}
\end{align}
where $R$ is a regularising operator in $X$, depending smoothly on $t$, and
\[
L_j = D_t - (\tau_j)_{\mathrm{op}}(t),
\]
where each $(\tau_j)_{\mathrm{op}}(t)$ is a classical pseudo-differential operator of order
one with principal symbol $\tau_j$, depending smoothly on $t$. It follows that the Cauchy problem~\eqref{eq:5.4.12} is well posed and that its solution is given by a composition of solution operators for the
problem~\eqref{AP1} with $A = (\tau_j)_{\mathrm{op}}$, which is a Cauchy problem for the operator $L_j$. 
\subsubsection{Well-posedness properties}
According to the above discussion and the composition formula for Fourier integral operators with complex-valued phase
functions, \cite{Melin-Sjostrand1976} we obtain the following result.

\begin{theorem}\label{XXX}
Let $P$ be a classical pseudo-differential operator of order $m$ of the form
\eqref{eq:5.4.10}. Assume that $P$ has simple characteristics $\tau_j,$ satisfying
$\operatorname{Im}\tau_j (t,x,\xi) \geq 0$ in $[0,T]\times (T^*X\setminus\{0\})$. Let $1\le p \le \infty$, and set
\[
\alpha_p \geq  \varkappa\left|\frac{1}{p}-\frac{1}{2}\right|.
\]
Let $\alpha \in \mathbb{R}$. Let the functions $f_l \in \mathcal{F}L^{p}_{\alpha+\alpha_p-l}$ be compactly supported for all
$0\le l \le m-1$. Then, for a fixed $t$, the solution $v$ of the Cauchy problem
\eqref{eq:5.4.12} satisfies
\[
v(t,\cdot)\in \mathcal{F}L^{p}_{\alpha}.
\]
\end{theorem}
\begin{proof}
    Note that the operator $P$ admits the factorisation in \eqref{Lm}. In consequence the solution $v$ can be written as
    \begin{equation}\label{FIO:representation}
        v(t,\cdot)=\sum_{j=1}^m\sum_{l=0}^{m-1}T_t^{jl}f_l
    \end{equation} where each $T_t^{jl}$ is a Fourier integral operator of order $-l.$ In consequence we have that 
    \begin{align*}
        \Vert v(t,\cdot) \Vert_{\mathcal FL^p_\alpha}&\leq \sum_{j=1}^m\sum_{l=0}^{m-1} \Vert T_t^{jl}f_l \Vert_{\mathcal FL^p_\alpha}\\
        &= \sum_{j=1}^m\sum_{l=0}^{m-1} \Vert T_t^{jl}(1-\Delta_x)^{\frac{-\alpha-\alpha_p+l}{2}}(1-\Delta_x)^{\frac{\alpha+\alpha_p-l}{2}}f_l \Vert_{\mathcal FL^p_\alpha}\\
        &= \sum_{j=1}^m\sum_{l=0}^{m-1} \Vert (1-\Delta_x)^{\frac{\alpha}{2}} T_t^{jl}(1-\Delta_x)^{\frac{-\alpha-\alpha_p+l}{2}}(1-\Delta_x)^{\frac{\alpha+\alpha_p-l}{2}}f_l \Vert_{\mathcal FL^p}.
    \end{align*} In view of the calculus of Fourier integral operators, their compositions with pseudo-differential operators on the left and on the right leave invariant the canonical relation. Then, using that each
    $$ (1-\Delta_x)^{\frac{\alpha}{2}} T_t^{jl}(1-\Delta_x)^{\frac{-\alpha-\alpha_p+l}{2}},$$ is a Fourier integral operator with complex phase, and of order
    \[
-\alpha_p \leq - \varkappa\left|\frac{1}{p}-\frac{1}{2}\right|.
\]
By applying Theorem~1.1 of \cite{CardonaObengDentehOpoku}, $(1-\Delta_x)^{\frac{\alpha}{2}} T_t^{jl}(1-\Delta_x)^{\frac{-\alpha-\alpha_p+l}{2}}$ is bounded from $\mathcal{F}L^p$ to itself, leading to the estimate
    \begin{align*}
        \Vert (1-\Delta_x)^{\frac{\alpha}{2}} T_t^{jl}(1-\Delta_x)^{\frac{-\alpha-\alpha_p+l}{2}}(1-\Delta_x)^{\frac{\alpha+\alpha_p-l}{2}}f_l \Vert_{\mathcal FL^p} &\lesssim \Vert (1-\Delta_x)^{\frac{\alpha+\alpha_p-l}{2}}f_l\Vert_{\mathcal FL^p}\\
        &=\Vert f_l\Vert_{\mathcal FL^p_{\alpha+\alpha_p-l}}.
    \end{align*} The proof is complete since  the functions $f_l \in \mathcal{F}L^{p}_{\alpha+\alpha_p-l}$ and then, for any $l,$ one has that $\Vert f_l\Vert_{\mathcal FL^p_{\alpha+\alpha_p-l}}<\infty.$ In consequence $ \Vert v(t,\cdot) \Vert_{\mathcal FL^p_\alpha}<\infty.$

In order to finish the proof we need to prove the existence of the representation in \eqref{FIO:representation}. This representation is well known for the real characteristics \cite[Page 108]{Ruzhansky2001}. We make here the extension for complex characteristics.  We proceed by induction. For $m=1,$ we need to solve the problem 
\begin{equation}\label{eq:5.4.12:m:1}
\begin{cases}
(D_t-\tau_1(x,D_x,t))v = 0, & t\in[0,T], \\
 v(0,x) = f_0(x), &
\end{cases} 
\end{equation}with $\tau_1(x,D_x,t)$ having a symbol with non-negative imaginary part. We are in the case of \eqref{AP1} and then the solution $v$ admits the representation $v(x,t)=U(t)f_0(x)$ with $f_0\in \mathcal{F}L^p_\alpha.$ Since $U(t)$ is, modulo $C^\infty,$ a Fourier integral operator of order zero with complex phase $\tau_1(x,\xi,t),$ and the proof of this case is complete.  

Now let us prove the general case. First let us make a remark. By the assumption of simple characteristics, the operator $P$ admits the factorization
\[
P = L_m \cdots L_1 + R,
\qquad
L_j = D_t - (\tau_j)_{\mathrm{op}}(t),
\]
with $R$ smoothing in $x$.
Since $R$ is regularizing, it suffices to construct a solution to
\[
L_m \cdots L_1 v = 0
\]
satisfying the Cauchy data, modulo smooth terms. For each $j$, let $U_j(t)$ denote the solution operator of the Cauchy problem
\[
\begin{cases}
L_j u = 0,\\
u(0)=g.
\end{cases}
\]
By the theory of Fourier integral operators with complex phase
(Melin--Sj\"ostrand \cite{melin2006fourier}), $U_j(t)$ is a Fourier integral operator of order $0,$ and satisfying
\[
\operatorname{Im}\varphi_j \ge 0.
\]
Define
\[
E_j(t)
:=
U_j(t)\circ U_{j-1}(t)\circ \cdots \circ U_1(t),
\qquad 1\le j\le m.
\]
We claim that any solution $v$ of
\[
L_m \cdots L_1 v = 0,
\]
can be written in the form
\[
v(t)
=
\sum_{j=1}^m E_j(t)\, g_j
\quad \text{mod } C^\infty,
\]
for suitable distributions $g_j$. This follows by induction on the number of factors $m.$
The case $m=1$  has been proved above.
Assuming the statement holds for $m-1$, one solves
\[
L_m w = 0
\quad \text{with} \quad
w = L_{m-1}\cdots L_1 v,
\]
and then applies the inductive hypothesis to recover $v$.
We now impose the initial conditions. 
Taking derivatives in time at $t=0$ yields
\[
\partial_t^l v(0)
=
\sum_{j=1}^m \partial_t^l E_j(0)\, g_j,
\qquad 0\le l\le m-1.
\]
In consequence, the previous is a system corresponding to the initial data $f_\ell.$
Taking principal symbols, we obtain the matrix
\[
M(\xi)
=
\bigl( ( i \tau_j(x,\xi) )^{l} \bigr)_
{0 \le l \le m-1 \atop 1 \le j \le m}.
\]
This is exactly a Vandermonde matrix:
\[
M =
\begin{pmatrix}
1 & 1 & \cdots & 1 \\
i\tau_1 & i\tau_2 & \cdots & i\tau_m \\
\vdots & \vdots & & \vdots \\
(i\tau_1)^{m-1} & (i\tau_2)^{m-1} & \cdots & (i\tau_m)^{m-1}
\end{pmatrix}.
\]Its determinant is
\[
\det M
=
\prod_{1 \le j < k \le m}
\bigl( i\tau_k - i\tau_j \bigr),
\]
which is microlocally non-vanishing, since the roots
$\tau_j$ are simple and pairwise distinct. 
Therefore, there exist pseudo-differential operators $B_{jl}$ of order $-l$
such that
\[
g_j
=
\sum_{l=0}^{m-1} B_{jl} f_l.
\]
Substituting into the representation of $v$, we obtain
\[
v(t)
=
\sum_{j=1}^m \sum_{l=0}^{m-1}
E_j(t)\, B_{jl}\, f_l.
\]
We now define
\[
T_t^{jl}
:=
E_j(t)\, B_{jl}.
\]
Each $E_j(t)$ is a Fourier integral operator of order $0$, and each
$B_{jl}$ is a pseudo-differential operator of order $-l$.
Hence, $T_t^{jl}$ is a Fourier integral operator of order $-l$ with complex phase.
This proves the claimed representation~\eqref{FIO:representation}.    
\end{proof}
The orders $\alpha_p$ are in general sharp, because in the case of strictly hyperbolic equations with $\operatorname{Im}\tau_j=0$, they can be shown to be
optimal by an application of the stationary phase method, under the condition that the projection from $\mathcal{C}_t$ (the canonical relation of the solution operator $U(t)$ of~\eqref{AP2})
to the base space has dimension $2d-1$ for at least one of the problems
\eqref{AP2} with $A=(\tau_j)_{\mathrm{op}}$. For more details we refer the reader to \cite{Nicola2010}.

\begin{theorem}
Let $P$ be as in Theorem \ref{XXX}. Let $1\le q < p\le \infty,$ and define
\[
\alpha_{pq}>d\bigg|\dfrac{1}{q}-\dfrac{1}{2}\bigg|+d\bigg(\dfrac{1}{q}-\dfrac{1}{p}\bigg).
\]
Then, for compactly supported Cauchy data
$f_l \in \mathcal{F}L^{p}_{\alpha+\alpha_{pq}-l}$, $0\le l \le m-1$, the fixed-time solution
$v(t,\cdot)$ of the Cauchy problem~\eqref{eq:5.4.12} satisfies
\[
v(t,\cdot)\in \mathcal{F}L^{q}_{\alpha}.
\]
\end{theorem}
\begin{proof}
    This follows from Theorem~(\ref{XXX}) that the solution $v$ can be expressed as
    \begin{equation}
        v(t,\cdot)=\sum_{j=1}^m\sum_{l=0}^{m-1}T_t^{jl}f_l
    \end{equation} where each $T_t^{jl}$ is a Fourier integral operator of order $-l.$ In consequence we have that 
    \begin{align*}
        \Vert v(t,\cdot) \Vert_{\mathcal FL^q_\alpha}&\leq \sum_{j=1}^m\sum_{l=0}^{m-1} \Vert T_t^{jl}f_l \Vert_{\mathcal FL^q_\alpha}\\
        &= \sum_{j=1}^m\sum_{l=0}^{m-1} \Vert T_t^{jl}(1-\Delta_x)^{\frac{-\alpha-\alpha_{pq}+l}{2}}(1-\Delta_x)^{\frac{\alpha+\alpha_{pq}-l}{2}}f_l \Vert_{\mathcal FL^q_\alpha}\\
        &= \sum_{j=1}^m\sum_{l=0}^{m-1} \Vert (1-\Delta_x)^{\frac{\alpha}{2}} T_t^{jl}(1-\Delta_x)^{\frac{-\alpha-\alpha_{pq}+l}{2}}(1-\Delta_x)^{\frac{\alpha+\alpha_{pq}-l}{2}}f_l \Vert_{\mathcal FL^q}.
    \end{align*}
    In view of the calculus of Fourier integral operators, their compositions with pseudo-differential operators on the left and on the right leave invariant the canonical relation. Then, using that each
    $$ (1-\Delta_x)^{\frac{\alpha}{2}} T_t^{jl}(1-\Delta_x)^{\frac{-\alpha-\alpha_{pq}+l}{2}},$$ is a Fourier integral operator with complex phase, and of order \[
-\alpha_{pq}<-d\bigg|\dfrac{1}{q}-\dfrac{1}{2}\bigg|-d\bigg(\dfrac{1}{q}-\dfrac{1}{p}\bigg).
\]
By applying Theorem~1.1 of \cite{CardonaObengDentehOpoku}, $(1-\Delta_x)^{\frac{\alpha}{2}} T_t^{jl}(1-\Delta_x)^{\frac{-\alpha-\alpha_{pq}+l}{2}}$ is bounded from  $\mathcal{F}L^p$ into $\mathcal{F}L^q,$ leading to the estimate
    \begin{align*}
        \Vert (1-\Delta_x)^{\frac{\alpha}{2}} T_t^{jl}(1-\Delta_x)^{\frac{-\alpha-\alpha_{pq}+l}{2}}(1-\Delta_x)^{\frac{\alpha+\alpha_{pq}-l}{2}}f_l \Vert_{\mathcal FL^q} &\lesssim \Vert (1-\Delta_x)^{\frac{\alpha+\alpha_{pq}-l}{2}}f_l\Vert_{\mathcal FL^q}\\
        &=\Vert f_l\Vert_{\mathcal FL^q_{\alpha+\alpha_{pq}-l}}.
    \end{align*} The proof is complete since  the functions $f_l \in \mathcal{F}L^{q}_{\alpha+\alpha_{pq}-l}$ and then, for any $l,$ one has that $\Vert f_l\Vert_{\mathcal FL^q_{\alpha+\alpha_{pq}-l}}<\infty.$ In consequence $ \Vert v(t,\cdot) \Vert_{\mathcal FL^q_\alpha}<\infty.$
\end{proof}

\begin{remark}\label{Fredferick:Remark}
    Please notice that setting $\alpha=0$ and $r'=p,$  we have that  $r=p'$ and from Theorem~\ref{Main:Theorem}, $s=0$ whenever $p'=\frac{p}{p-1}$ and $q'=\frac{q}{q-1},$ indicating that $\mathscr{F}L^{r'}_{\alpha} \hookrightarrow B^{-s}_{p',q'}, $ which implies that the embedding $\mathscr{F}L^{p} \hookrightarrow B^{0}_{p',q'}$ holds. Similarly, $\mathscr{F}L^{r'}_{\alpha} \hookrightarrow F^{-s}_{p',q'}, $ implies that  $\mathscr{F}L^{p} \hookrightarrow F^{0}_{p',q'}$ whenever $p'=\frac{p}{p-1}$ and $q'=\frac{q}{q-1}$. Based on these assertions, it suffices to say that, if $v(t,\cdot)\in \mathcal{F}L^{p}$ then $v(t,\cdot)\in  B^{0}_{p',q'}$ and $ v(t,\cdot) \in F^{0}_{p',q'}$.
\end{remark}
    
\noindent {\bf Acknowledgement.}   Duv\'an Cardona has been  supported  by the FWO  Odysseus  1  grant  G.0H94.18N:  Analysis  and  Partial Differential Equations and by the Methusalem programme of the Ghent University Special Research Fund (BOF)
(Grant number 01M01021), by the FWO Fellowship
Grant No 1204824N and by the FWO Grant K183725N of the Belgian Research Foundation FWO. Frederick Opoku has been supported by Kwame Nkrumah University of Science and Technology, (KNUST) of Ghana.

\bibliographystyle{amsplain}

\end{document}